\documentclass[12pt,a4paper]{article}
\usepackage{epsfig,psfrag,amsmath,amssymb,latexsym}
\usepackage{amscd }
\usepackage{amsfonts}
\usepackage{graphicx}
\usepackage{epstopdf}
\usepackage{bbm}
\usepackage{sectsty}
\usepackage{amsthm}
\usepackage[numbers]{natbib}
\usepackage{dsfont}

\sectionfont{\large}

\def\P{\mathbb{P}}
\def\Q{\mathbb{Q}}
\def\E{\mathbb{E}}
\def\R{\mathbb{R}}
\def\N{\mathbb{N}}

\def\R{\mathbb{R}}
\DeclareMathOperator{\arsinh}{\mathrm{arsinh}}
\DeclareMathOperator{\arcosh}{\mathrm{arcosh}}
\DeclareMathOperator{\dilog}{\mathrm{dilog}}
\DeclareMathOperator{\Defi}{\mathrel{\mathop:}=}
\def\h{\mathbf{h}}
\newcommand{\artanh}{\operatorname{artanh}}

\newtheorem{theorem}{Theorem}[section]

\newtheorem{definition}{Definition}[section]

\theoremstyle{remark}
\newtheorem{remark}{Remark}[section]
\newtheorem{example}{Example}[section]

\numberwithin{equation}{section}

\title{\large LARGE DEVIATIONS PRINCIPLE FOR CURIE-WEISS MODELS WITH RANDOM FIELDS}
\author{Matthias L\"owe, Raphael Meiners\thanks{Research supported by German
National Academic Foundation} \,and Felipe Torres\thanks{Corresponding author: ftorrestapia@uni-muenster.de}\thanks{Research supported by DFG through SFB 878 at University of M\"unster}
\\
{\small \it Institute for Mathematical Statistics, University of M\"unster, Germany}}

\begin{document}
\maketitle
\vspace{-12pt}
\abstract{\noindent In this article, we consider an extension of the classical Curie-Weiss model in which the global and deterministic external magnetic field is replaced by local and random external fields which interact with each spin of the system. We prove a Large Deviations Principle for the so-called {\it magnetization per spin} $S_n/n$ with respect to the associated Gibbs measure, where $S_n/n$ is the scaled partial sum of spins. In particular, we obtain an explicit expression for the rate function, which enables an extensive study of the phase diagram in some examples. It is worth mentioning that the model considered in this article covers, in particularly, both the case of i.\,i.\,d.\ random external fields (also known under the name of random field Curie-Weiss models) and the case of dependent random external fields generated by e.\,g.\ Markov chains or dynamical systems.
\\
\\
{\bf Keywords: } Disordered mean-field model, Large Deviations Principle, Random field Curie-Weiss model
\\
\\
{\bf AMS 2010 subject classifications: 60F10, 82B44}
%%%%%%%%%%%%%%%%%%%%%%%%%%%%%%%%%%%%%%%%%%%%%%%%%%%%%%%%%%%%%%%%%%%%%%%%
\section{Introduction}
Curie-Weiss models belong to the class of models in statistical mechanics for which one can explicitly explain important physical phenomena such as multiple phases, metastable states and, particularly, how macroscopic ob\-ser\-vables fluctuate around their mean values when close to or at critical temperatures. Ellis and Newman (see \cite{Elli1978b, Elli1978a, Elli1980}) computed the statistics of large spin-block variables in the class of {\it classical} Curie-Weiss models, which correspond to a finite collection of $n$ (spin) random variables with an equal interaction of strength $1/n$ between each pair of spins, in the presence of a constant and deterministic external magnetic field. They derived a Law of Large Numbers and a Central Limit Theorem for $S_n/n$
(where $S_n$ is the total magnetization of the model with $n$ spins), basically showing that $S_n/\sqrt{n}$ has a non-Gaussian limit at a critical temperature $\beta_c$, and a Gaussian limit otherwise. From a probabilistic point of view, a natural next step after showing a Law of Large Numbers and a Central Limit Theorem would be to study large and moderate deviations properties of $S_n/n$. While a Large Deviations Principle (LDP) already goes back to Ellis \cite{Elli1985}, in \cite{Eich2004} Eichelsbacher and L\"owe showed that $S_n/n$ satisfies Moderate Deviations Principles (MDPs) with respect to the associated Gibbs measure.
\medskip
\\
In this paper, we are interested in the behaviour of $S_n/n$ for so-called {\it Curie-Weiss models with random fields}, which are derived from the Curie-Weiss model by replacing the constant external field by local and random external fields which interact with each spin of the system. Noteworthy exponents of these models are the {\it random field Curie-Weiss models} (abbreviated by RFCW models), where the local external fields are given by i.\,i.\,d.\ random variables, making RFCW models one of the easiest disordered mean-field models. These models have been studied intensively over the last decades, see e.\,g.\ \cite{Schn1977,Ahar1978,Sali1985,Amar1991,Amar1992} and the references therein. It maybe worthwhile noting that many of the above authors assume the local external fields to be symmetrically Bernoulli distributed or to be bounded, while we do not need these restricting assumptions to consider RFCW models as representatives of Curie-Weiss models with random fields. In the setting of RFCW models, Amaro de Matos and Perez showed in \cite{Amar1991} a Central Limit Theorem for $S_n/n$, by using techniques similar to the ones employed by Ellis and Newman. In this fashion, they also proved that $S_n/\sqrt{n}$ has a Gaussian or a non-Gaussian limit (depending on temperature), though the main difference with the classical setting is that these Gaussian and non-Gaussian limits depend on the realization of the local random external field. A natural next step, as before, would be to study an LDP and MDPs for $S_n/n$ with respect to the now random Gibbs measure (disorder dependent). Almost sure MDPs have been proved in \cite{Loe2012} and an almost sure LDP can, in principle, be derived for RFCW models with bounded external fields using a result of Comets (Theorem V (ii) in \cite{Come1989}). As hinted at, we can prove an LDP for RFCW models even if the external fields are unbounded. Yet, this is not the only reason why we want to prove an LDP in this article. The rate function obtained by using the result of Comets is extremly abstract and, hence, makes a study of the phase diagram impossible. We prove an LDP with an explicit rate function and make use of this advantage by studying the phase diagram in some examples. It should be mentioned that for certain choices of the random external field the analysis of the phase diagram for the corresponding RFCW model had been already obtained by using techniques related to the study of the dynamics and the study of the metastates for the RFCW model (see e.\,g.\ \cite{Font2000,Bovi2009}) and for more general mean-field models (see e.\,g.\ \cite{Iaco2010,Kuel1997}). Moreover, our approach also yields an LDP if the external magnetic fields are \emph{not} independent.
\medskip
\\
The rest of the paper is organized as follows: In Section \ref{sec1} we mathematically describe the Curie-Weiss models with random fields. Our main result will be stated and proved in Section \ref{sec2}. Finally, Section \ref{sec3} is devoted to applications of our result and their use to study phase diagrams.

%%%%%%%%%%%%%%%%%%%%%%%%%%%%%%%%%%%%%%%%%%%%%%%%%%%%%%%%%%%%%%%%%%%%%%%%
\section{Curie-Weiss models with random fields}\label{sec1}
The Curie-Weiss model is a mean-field model of a ferromagnet. Due to its mean-field structure the spatial location of the spins is unimportant. The Hamiltonian of the Curie-Weiss model with external magnetic field $h \in \R$ can, therefore, be described by
\begin{equation}\label{HamiltonianCW}
H_{n,h}^{\text{CW}} (\sigma)~=~ -\frac 1 {2n} \sum_{i,j=1}^n  \sigma_i \sigma_j - h \sum_{i=1}^n \sigma_i ~=~ -\frac 1 {2n} S_n(\sigma)^2 -h S_n(\sigma), \quad \sigma\in\{-1,+1\}^n ,
\end{equation}
where $S_n(\sigma)=\sum_{i=1}^n \sigma_i $ is the total magnetization of the model. In the Curie-Weiss model with random fields the global external magnetic field $h$ is now replaced by local and random external magnetic fields $(\h_i, i \in \N)$, defined on some probability space $(\Omega, \mathcal{F}, P)$, which satisfy $P$-almost surely
\begin{equation}\label{f}
f_n(x) ~\Defi~ \frac1n \sum_{i=1}^n \ln \cosh (x+\beta \h_i) ~\longrightarrow~ f(x)
\end{equation}
as $n \to \infty$ for every $x \in \Q$ and some differentiable function $f: \R \rightarrow \R$. Its Hamiltonian is, thus, given by
\begin{equation*}
H_n^{\h} (\sigma)~=~ -\frac 1 {2n} \sum_{i,j=1}^n  \sigma_i \sigma_j - \sum_{i=1}^n \h_i\sigma_i,  \qquad \sigma\in\{-1,+1\}^n.
\end{equation*}
Correspondingly, the model can be associated with the following Gibbs measure on $\{-1,+1\}^n$ at inverse temperature $\beta >0$
\begin{equation}\label{gibssm}
P_{n,\beta}^{\h}(\sigma) ~=~  \frac{1}{Z_{n,\beta}^\h}\exp\left(\frac{\beta}{2n} \sum_{i,j=1}^n \sigma_i \sigma_j + \beta\sum_{i=1}^n \h_i \sigma_i\right)
\end{equation}
where
\begin{equation*}
Z_{n,\beta}^{\h} ~=~ \sum _{\sigma \in \{-1,+1\}^n} \exp\left(\frac{\beta}{2n} \sum_{i,j=1}^n \sigma_i \sigma_j + \beta\sum_{i =1}^n \h_i \sigma_i\right)
\end{equation*}
is the partition function of the model.

\noindent We can see that, other than in the ordinary Curie-Weiss model, $S_n$ is not an order parameter for the Curie-Weiss model with random fields, since in expression (\ref{gibssm}) the measure is not completely determined by the value of $S_n$, now that the local fields also plays an important role.

\noindent We want to end this section with three examples, which shall illustrate the diverse nature of Curie-Weiss models with random fields.

\begin{example}[Random field Curie-Weiss models]\label{exam:RFCW}
In this first example, the external fields are assumed to be i.\,i.\,d.\ random variables, that is, $\h$ is distributed according to the $\N$-fold product measure of the marginal distribution of $\h_1$, $\nu$. Assume that $\nu$ has a finite absolute first moment, i.\,e.\,
\begin{equation*}
\int_{h \in \R}|h|\, d\nu(h) ~<~ \infty.
\end{equation*}
Since $|\ln \cosh (x)| \leq |x|$, \eqref{f} is, by the Law of Large Numbers, satisfied with
\begin{equation*}
f(x) ~=~ \int_{h \in \R} \ln \cosh (x + \beta h) \,d\nu(h).
\end{equation*}
\end{example}

\begin{example}[External fields generated by Markov chains]
In this example, the external fields are still assumed to be identically distributed, but the independence assumption is dropped. To be more precise, let $\h$ be an irreducible Markov chain on a countable state space $S$ that has a stationary distribution $\pi$. Moreover, let us assume that $\h$ is in equilibrium, that is, the initial state $\h_1$ is distributed according to $\pi$, and that $\pi$ has a finite absolute first moment.
Then, with the help of the ergodic theorem (cf.\ Example 7.2.2 in \cite{Dur2010}), \eqref{f} is satisfied with
\begin{equation*}
f(x) ~=~ \sum_{h \in S} \ln \cosh (x + \beta h) \, \pi(h).
\end{equation*}
It should be noted that this function does not depend on the transition probabilities of the Markov chain, but only on the stationary distribution $\pi$.
\end{example}

\begin{example}[External fields generated by dynamical systems]
Let $(\Omega, \mathcal{F}, P, \phi)$ be a dynamical system, that is, $(\Omega, \mathcal{F}, P)$ be the underlying probability space and $\phi : \Omega \rightarrow \Omega$ be a measure preserving transformation. Assume that $\phi$ is ergodic and that $X:\Omega\rightarrow \R$ is a $\mathcal{F}$-measurable random variable with finite absolute first moment. By the ergodic theorem (Theorem 7.2.1 in \cite{Dur2010}), \eqref{f} is satisfied for $(\h_i)_{i \in \N} = (X \circ \phi^i)_{i \in \N_0}$ with
\begin{equation*}
f(x) ~=~ \int_{h \in \R} \ln \cosh (x + \beta h)\, dP^X(h).
\end{equation*}
\end{example}

%%%%%%%%%%%%%%%%%%%%%%%%%%%%%%%%%%%%%%%%%%%%%%%%%%%%%%%%%%%%%%%%%%%%%%%%%%
\section{Main result}\label{sec2}
Before we state our result on an LDP for $(S_n/n)_n$, let us recall the basic definition of an LDP. We will just consider the situation where elements of $\mathcal{B}({\R})$, the Borel $\sigma$-field on $\R$, are of interest. For more general definitions see \cite{Demb1998}.

\begin{definition}
Let $(\mathds{P}_n)_{n \in \N}$ be a sequence of probability measures. A sequence of real-valued random variables $(X_n)_{n \in \N}$ is said to satisfy an LDP w.\,r.\,t.\ $(\mathds{P}_n)_{n\in\N}$ with rate function $I : \R \to [0,\infty]$ if
\begin{itemize}
\item $I$ has compact level sets $\{x \in \R: I(x) \leq c\}\subset \R$ for any $c \in \R$,
\item (Upper bound) for every closed set $C \in \mathcal{B}({\R})$ it holds
\begin{equation}\label{upperLDP}
\limsup_{n \to \infty} \frac{1}{n} \ln \mathds{P}_n (X_n \in C) ~\le~ - \inf_{x \in C} I(x),
\end{equation}
\item (Lower bound) for every open set $O \in \mathcal{B}({\R})$ it holds
\begin{equation}\label{lowerLDP}
\liminf_{n\to \infty} \frac{1}{n} \ln \mathds{P}_n (X_n \in O) ~\ge~ - \inf_{x \in O} I(x).
\end{equation}
\end{itemize}
\end{definition}

\noindent Note that in this definition and the rest of this article the infimum of a function over an empty set is interpreted as $\infty$. Also, what is called a rate function in this definition is sometimes referred to as a \emph{good} rate function in literature. Having said this, we can state and prove our main result:

\begin{theorem}\label{LDPRFCW}
The magnetization $(S_n/n)_{n \in \N}$ satisfies $P$-almost surely an LDP w.\,r.\,t. $\P_{n,\beta}^{\h}$ with rate function $I(x) = J(x) -\inf_{y\in\R}J(y)$, where $J(x) = f^{*}(x)-\beta x^2/2$. Therein, $f^*$ denotes the Legendre-Fenchel transform of $f$, that is,
\begin{equation*}
f^*(x) ~\Defi~ \sup_{y\in\R}\{xy-f(y)\}.
\end{equation*}
\end{theorem}

\begin{proof}
With probability one, we can choose a realization $h = (h_i)_{i \in \N}$ of $\h$ that satisfies \eqref{f}. This realization satisfies \eqref{f} not just for $x \in \Q$, but also for any $x \in \R$. Indeed, for $x \in \R$ and every $\varepsilon > 0$ we can choose $y \in \Q$ such that
\begin{eqnarray}
|x-y| &\leq& \frac{\varepsilon}{3} \text{ and } \label{eq:lipschitz}\\
|f(x) - f(y)| &\leq& \frac{\varepsilon}{3}\label{2},
\end{eqnarray}
since $\Q$ is dense in $\R$ and $f$ is a continuous function. Note that for every $n \in \N \ f_n$ is Lipschitz continuous with constant 1, since for $x_1,x_2 \in \R, x_1 \geq x_2$,
\begin{eqnarray*}
&& |f_n(x_1) - f_n(x_2)| -|x_1-x_2| \\
%&=& \frac{1}{n} \left|\sum_{i=1}^n\big(\ln \cosh(x_1+\beta h_i)-\ln \cosh(x_2+\beta h_i)\big)\right| -x_1+x_2\\
&\leq& \frac{1}{n} \sum_{i=1}^n\left|\ln \left( \frac{\cosh(x_1+\beta h_i)}{\cosh(x_2+\beta h_i)}\right)\right|-x_1+x_2 \\
&=& \frac{1}{n} \sum_{i=1}^n \max\left\{\ln \left(\frac{\cosh(x_1+\beta h_i)}{\cosh(x_2+\beta h_i)}\right),\ln \left(\frac{\cosh(x_2+\beta h_i)}{\cosh(x_1+\beta h_i)}\right)\right\} +\ln e^{-x_1+x_2}\\
&\leq& \frac{1}{n} \sum_{i=1}^n \max\left\{\ln \left(\frac{e^{x_1+x_2+\beta h_i}+e^{x_2-x_1-\beta h_i}}{e^{x_1+x_2+\beta h_i}+e^{x_1-x_2-\beta h_i}}\right), \ln \left(\frac{e^{x_2-x_1+\beta h_i}+e^{-x_1-x_2-\beta h_i}}{e^{x_1-x_2+\beta h_i}+e^{-x_1-x_2-\beta h_i}}\right)\right\}\le 0\\
%&\leq& 0.
\end{eqnarray*}
Therefore, \eqref{eq:lipschitz} implies
\begin{equation}\label{3}
|f_n(x)-f_n(y)| ~\leq~ \frac{\varepsilon}{3}.
\end{equation}
Since $y \in \Q$, we have $|f_n(y)-f(y)| \leq \varepsilon /3$ for $n$ sufficiently large by assumption and, consequently,
%\begin{equation*}
$|f_n(x) - f(x)| ~\leq~ |f_n(x) - f_n(y)| + |f_n(y) - f(y)| + |f(y)- f(x)| ~\leq~ \varepsilon$,
%\end{equation*}
where we have made use of \eqref{2} and \eqref{3}. Thus, we can assume the validity of \eqref{f} for any $x \in \R$.

\noindent Next, let $Q_i$ for every $i \in \N$ be the probability measure on $\{-1,+1\}$
%$(\{-1,1\},\mathcal{P}(\{-1,1\}))$ 
induced by
\begin{equation*}
Q_i(\{1\}) ~\Defi~ \frac{e^{\beta h_i}}{2 \cosh(\beta h_i)}
\end{equation*}
and denote by $Q$ the product measure $\otimes_{i=1}^\infty Q_i$. As a first step towards an LDP for $S_n/n$ w.\,r.\,t.\ $P_{n, \beta}^h$ we prove an LDP w.\,r.\,t.\ $Q$, which will be done by means of the G\"artner-Ellis Theorem (cf.\ Theorem 2.3.6 in \cite{Demb1998}). To this end, we compute the logarithmic moment generating function $\Lambda$ and observe that for $x \in \R$
\begin{eqnarray*}
\Lambda(x) &\Defi&
\lim_{n \to \infty} \frac1n \ln \E e^{n x S_n/n}\\
&=& \lim_{n \to \infty} \frac1n \sum_{i=1}^n \ln \left(\frac{e^{x + \beta h_i} + e^{-x-\beta h_i}}{2 \cosh(\beta h_i)}\right)\\
&=& \lim_{n \to \infty} \big(f_n(x)-f_n(0)\big)\\
&=& f(x)-f(0),
\end{eqnarray*}
where $\E$ denotes the expectation w.\,r.\,t.\ $Q$. In doing so, we have used that $S_n$ is a sum of independent random variables (under the product measure $Q$!) to derive the second line and that \eqref{f} holds for any $x \in \R$ to get the last line. Since $f$ is differentiable, the same holds for $\Lambda$ and the G\"artner-Ellis Theorem yields that $S_n/n$ satisfies an LDP w.\,r.\,t. $Q$ with rate function $R$ given by $R(x) \Defi f^{*}(x) + f(0)$. Note that this implies $f^*(x) = \infty$ for $|x| > 1$ since $(S_n/n)_{n\in \N}$ is uniformly bounded by $1$.

Next, we see that $P_{n,\beta}^h \circ (S_n/n)^{-1}$ is a tilted version of $Q\circ (S_n/n)^{-1}$, since with
\begin{equation*}
F : \R \rightarrow \R, \quad x ~\mapsto~
\begin{cases}
\beta x^2 / 2,& \text{ if } |x| \leq 1,\\
\beta /2,&\text{ otherwise,}
\end{cases}
\end{equation*}
and $\mathcal{X}_n = \{-1, -1+2/n, \ldots, 1-2/n, 1\}$ we have got for every $S \in \mathcal{B}(\R)$

\begin{eqnarray*}
P_{n,\beta}^h \circ (S_n/n)^{-1}(S)
&=& \sum_{m \in S \cap \mathcal{X}_n} P_{n,\beta}^h (S_n/n = m)\\
&=& \frac{\sum_{m \in S\cap \mathcal{X}_n} \sum_{\sigma: S_n(\sigma)=n m} e^{\frac{\beta}{2n}(\sum_{i=1}^n\sigma_i)^2+\beta \sum_{i=1}^n h_i \sigma_i}}{\sum_{m \in \mathcal{X}_n} \sum_{\sigma: S_n(\sigma) = n m} e^{\frac{\beta}{2n}(\sum_{i=1}^n\sigma_i)^2+\beta \sum_{i=1}^n h_i \sigma_i}}\\
&=& \frac{\sum_{m \in S\cap \mathcal{X}_n} e^{n F(m)} \sum_{\sigma: S_n(\sigma)=nm} \prod_{i=1}^n Q_i(\{\sigma_i\})}{\sum_{m \in S\cap \mathcal{X}_n} e^{n F(m)} \sum_{\sigma: S_n(\sigma)=nm} \prod_{i=1}^n Q_i(\{\sigma_i\})}\\
&=& \frac{\sum_{m \in S\cap \mathcal{X}_n} e^{n F(m)} Q \circ (S_n/n)^{-1}(m)}{\sum_{m \in \mathcal{X}_n} e^{n F(m)} Q \circ (S_n/n)^{-1}(m)}.
\end{eqnarray*}

\noindent The Tilted LDP (cf.\ Theorem III. 17 in \cite{DenH}) then yields that $S_n/n$ satisfies an LDP w.\,r.\,t.\ $P_{n, \beta}^h$ with rate function
\begin{eqnarray*}
I(x)
&=& R(x) - F(x) - \inf_{y\in\R}\{R(y) - F(y)\}\\
&=& f^*(x) - F(x) - \inf_{y\in\R}\{f^*(y) - F(y)\}\\
&=& f^*(x) - \beta x^2/2 - \inf_{y\in\R}\{f^*(y) - \beta y^2/2\},
\end{eqnarray*}
where we have used $f^*(x) = \infty$ for $|x| > 1$ to obtain the last line, which, in turn, is the desired form of the rate function.
\end{proof}

\begin{remark}
Looking at the previous proof, we see that we carried over the problem to an LDP w.\,r.\,t.\ $Q$ and proved this by means of the G\"artner-Ellis Theorem. One could have, instead, thought about directly using the G\"artner-Ellis Theorem to prove an LDP w.\,r.\,t.\ $P_{n,\beta}^h$ with rate function $I$. However, it is not only way more complicated to calculate the logarithmic moment generating function in that setting since $S_n$ is \emph{not} a sum of independent random variables under the measure $P_{n,\beta}^h$, but also does the logarithmic moment generating function in general not satisfy the assumptions of the G\"artner-Ellis Theorem. This is evident, since the desired rate function $I$ is in general not convex (cf.\ e.\,g.\ Example \ref{example2}) and the G\"artner-Ellis Theorem can only be used to prove LDPs with convex rate functions.
\end{remark}

\begin{remark}\label{remark:minima}
It is well-known from Large Deviations Theory (see e.\,g.\ Theorem II.7.2 in \cite{Elli1985}) that the magnetization per spin $S_n/n$ is asymptotically concentrated around the global minima of the rate function. Thus, it is important to know the precise structure of the rate function. The explicit representation of $I$ given in Theorem \ref{LDPRFCW} enables an accurate study of the global minima of the rate function $I$. Since $f$ is the pointwise limit of convex functions $(f_n)_{n \in \N}$, $f$ itself is convex and, thus, Theorem A.1 in \cite{Cost2005} yields that the global minima of $I$ coincide with the global minima of $G$ given by
\begin{equation}\label{defG}
G(x) ~\Defi~  G_{\beta}(x) ~\Defi~ \frac{\beta}{2}x^2-f(\beta x).
\end{equation}
Therefore, a study of $G$ reveals the phase diagram of Curie-Weiss models with random fields and we identify the following phases of the system:
\begin{itemize}
\item[(i)]{\it Paramagnetic phase}: $G$ has a unique global minimum of type 1.
\item[(ii)]{\it Ferromagnetic phase}: $G$ has two global minima, both of type 1.
\item[(iii)]{\it First-order phase transition}: $G$ has several global minima, all of type 1.
\item[(iv)]{\it Second-order phase transition}: $G$ has a unique global minimum of type 2.
\item[(v)]{\it Tricritical point}: $G$ has a unique global minimum of type 3.
\end{itemize}
In this process we said that a minimum $m$ is of type $k \in \N$ if there is a positive real number $\lambda$ (called strenght) such that
\begin{equation*}
G(x) ~=~ G(m) +\frac{\lambda}{(2k)!} (x-m)^{2k} +\mathcal{O}\big((x-m)^{2k+1}\big) \quad \text{ as } x\rightarrow m.
\end{equation*}
\end{remark}
%%%%%%%%%%%%%%%%%%%%%%%%%%%%%%%%%%%%%%%%%%%%%%%%%%%%%%%%%%%%%%%%%%%%%%%
\section{Examples}\label{sec3}
In this section we want to apply Theorem \ref{LDPRFCW} in some examples and study the corresponding phase diagram as elucidated in Remark \ref{remark:minima}:

\begin{subsection}{Classical Curie-Weiss model}
Assume $\h_i \equiv h$ for all $i \in \N$ and some $h \in \R$. In that case we have
\begin{equation*}
f(x) ~=~ \ln \cosh (x + \beta h)
\end{equation*}
and Theorem \ref{LDPRFCW} yields that the magnetization satiesfies an LDP with rate function
\begin{eqnarray*}
I(x)
&=& \sup_{y \in \R} \left\{x y - f(y)\right\} - \frac{\beta}{2} x^2 - \inf _{y \in \R}\left\{\sup_{z \in \R} \left\{y z - f(z)\right\}-\frac{\beta}{2} y^2\right\}.
\end{eqnarray*}
Using the identity
\begin{equation*}
\ln \cosh \artanh x ~=~ -\frac12 \ln (1-x^2),
\end{equation*}
which holds for all $|x|<1$, one sees that
\begin{equation*}
\sup_{y \in \R} \left\{x y - f(y)\right\} ~=~ I_0(x) - \beta h x,
\end{equation*}
where $I_0$ is given by
\begin{equation*}
I_0(x) ~\Defi~
\begin{cases}
\frac{1+x}{2} \ln(1+x)+\frac{1-x}{2}\ln(1-x)& \text{ if } |x|<1\\
\ln 2& \text{ if }|x| = 1\\
\infty& \text{ if }|x| > 1.
\end{cases}
\end{equation*}
Thus, we get
\begin{equation}\label{RatefunctionCW}
I(x) ~=~ -\frac{\beta}{2} x^2 - \beta h x + I_0(x) - \inf _{y \in \R}\left\{-\frac{\beta}{2} y^2 - \beta h y + I_0(y)\right\},
\end{equation}
which is the well-known representation of the rate function in the setting of the classical Curie Weiss model (see e.\,g.\ (4.17) in \cite{Elli1985}). We omit a discussion of the rate function and the phase diagram as they are well-discussed (see e.\,g.\ \cite{Elli1985}).
\end{subsection}

\begin{subsection}{Curie-Weiss model with dichotomous external fields}\label{example2} 

Assume that $(\h_i)_{i \in \N}$ is an i.\,i.\,d.\ sequence, where $\h_1$ takes the values $h$ and $-h$ with probability $1/2$ each, $h \in \R$. Again, Theorem \ref{LDPRFCW} yields that the magnetization satisfies an LDP with rate function $I$ given by (cf.\ Figure \ref{fig1}):
\begin{itemize}
\item[1.] If $|x| < 1$, then
\begin{eqnarray*}
I(x) &=& \ln 2 -\frac{\beta}{2} x^2 +\frac{x}{2} \arsinh\left(\frac{x}{1-x^2} \left(b+\sqrt{1+x^2 a^2}\right)\right) + \frac12 \ln\left(\frac{1-x^2}{2}\right)\\
&& - \frac12 \ln\big(b+\sqrt{1+x^2 a^2}\big)- \inf_{x \in \R}G(x)
\end{eqnarray*}
with $a \Defi \sinh(2 \beta h), b \Defi \cosh(2 \beta h)$ and (cf.\ \eqref{defG})
\begin{equation*}
G(x) ~=~ \frac{\beta}{2} x^2 - \frac{1}{2}\ln\left(\cosh [\beta (x + h)] \cosh [\beta (x - h)]\right).
\end{equation*}
\item[2.] If $|x| = 1$, then
\begin{equation*}
I(x) ~=~
\ln 2-\frac{\beta}{2} x^2 - \inf_{x \in \R}G(x).
\end{equation*}
\item[3.] If $|x|>1$, then
\begin{equation*}
I(x) ~=~ \infty.
\end{equation*}
\end{itemize}
\begin{figure}[!h]
\begin{center}
\includegraphics[width=12cm]{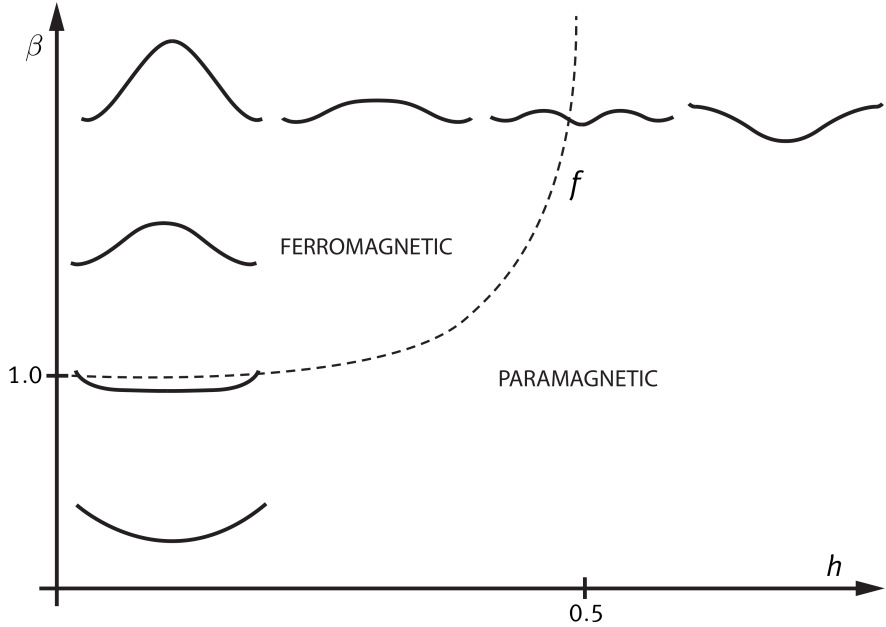}
\caption{\label{fig1} $I$ displays phase transitions as a function of $\beta$ and $h$.}
\end{center}
\end{figure}
By Remark \ref{remark:minima}, in order to obtain the phase diagram of the Curie-Weiss model with dichotomous external fields, it is enough to study $G$, which has been done in Chapter 5 of \cite{Amar1992}. We note a rich family of phase transitions in $\beta$ as well as in $h$: For $h \geq \frac12$ $G$ has a unique global minimum at $x=0$. In contrast to that the case $h < \frac12$ is more subtle. One finds a strictly increasing function $f: [0,\frac12) \rightarrow \R$ with $f(0) = 1$ and $f(x_n) \rightarrow \infty$ as $x_n \nearrow \frac12$ such that
\begin{itemize}
\item[(i)] $\beta < f(h)$:\\
$G$ has an unique global minimum at $x=0$.
\item[(ii)] $\beta = f(h)$:\\
$G$ has an unique global minimum at $x=0$ if $h \leq \frac23 \arcosh \sqrt{3/2}$. Otherwise it has three global minima, one at $x=0$ and the others symmetric.
\item[(iii)] $\beta > f(h)$:\\
$G$ has two symmetric global minima.
\end{itemize}
Consequently, we identify the following regions in the phase diagram:
\begin{itemize}
\item[(i)] If $h \geq \frac12$ or $h < \frac12, \beta < f(h)$, the system is in the \emph{paramagnetic phase}.
\item[(ii)] If $h < \frac12, \beta > f(h)$, the system is the \emph{ferromagnetic phase}.
\item[(iii)] The remaining region $h < \frac12, \beta = f(h)$ is again delicate as that critical line consists of three different regimes. If $h < h_c = \frac23 \arcosh \sqrt{3/2}$ we note a \emph{second-order phase transition}. At $h= h_c$ we note a \emph{tricritical point} and for $h > h_c = \frac23 \arcosh \sqrt{3/2}$ we observe a \emph{first-order phase transition}.
\end{itemize}
The same analysis could have been done in the setting of non-symmetric dichotomous external fields, that is, $\h_1$ takes the values $h$ and $-h$ with probabilities $\alpha$ and $1-\alpha$ respectively, $\alpha \in [0,1]$. In this case, we get
\begin{equation*}
G(x) ~=~ \frac{\beta}{2} x^2 - \alpha \ln \cosh (\beta (x + h))  -(1-\alpha) \ln \cosh (\beta (x - h)),
\end{equation*}
obtaining the same phase diagram as the one presented by Theorem 4.1 in \cite{Kuel2007}.
\end{subsection}

\begin{subsection}{Curie-Weiss model with uniformly distributed external fields}\label{lastexample}
In this last example, we want to study the phase diagram of the Curie-Weiss model with uniformly distributed external fields, that is, let $(\h_i)_{i \in \N}$ be an i.\,i.\,d.\ sequence, where $\h_1$ is uniformly distributed on the interval $[-h,h]$ for some $h> 0$. Again, by Remark \ref{remark:minima}, we need to study $G$, which in this setting is given by
\begin{equation}\label{Guniform}
G(x) ~=~ \ln 2 +\frac{\beta}{2} x^2- \frac{\dilog(-\exp(-2\beta(h+|x|)))}{4\,\beta \,h} + L(x)
\end{equation}
with
\begin{equation*}
L(x) ~=~
\begin{cases}
-\frac{\beta}{2 \,h}(x^2+h^2)-\frac{\pi^2}{24\, \beta\, h} - \frac{\dilog(-\exp(-2\beta(h-|x|)))}{4\,\beta\, h}& \text{ if } |x|\leq h\\
-\beta |x| + \frac{\dilog(-\exp(-2\beta(|x|-h)))}{4\,\beta \,h}& \text{ if } |x|> h.
\end{cases}
\end{equation*}
Therein, $\dilog$ denotes the dilogarithm that is defined by the power series
\begin{equation*}
\dilog(z) ~=~ \sum_{n=1}^{\infty}\frac{z^n}{n^2}.
\end{equation*}

\begin{figure}[h!]
\begin{center}
\includegraphics[width=12cm]{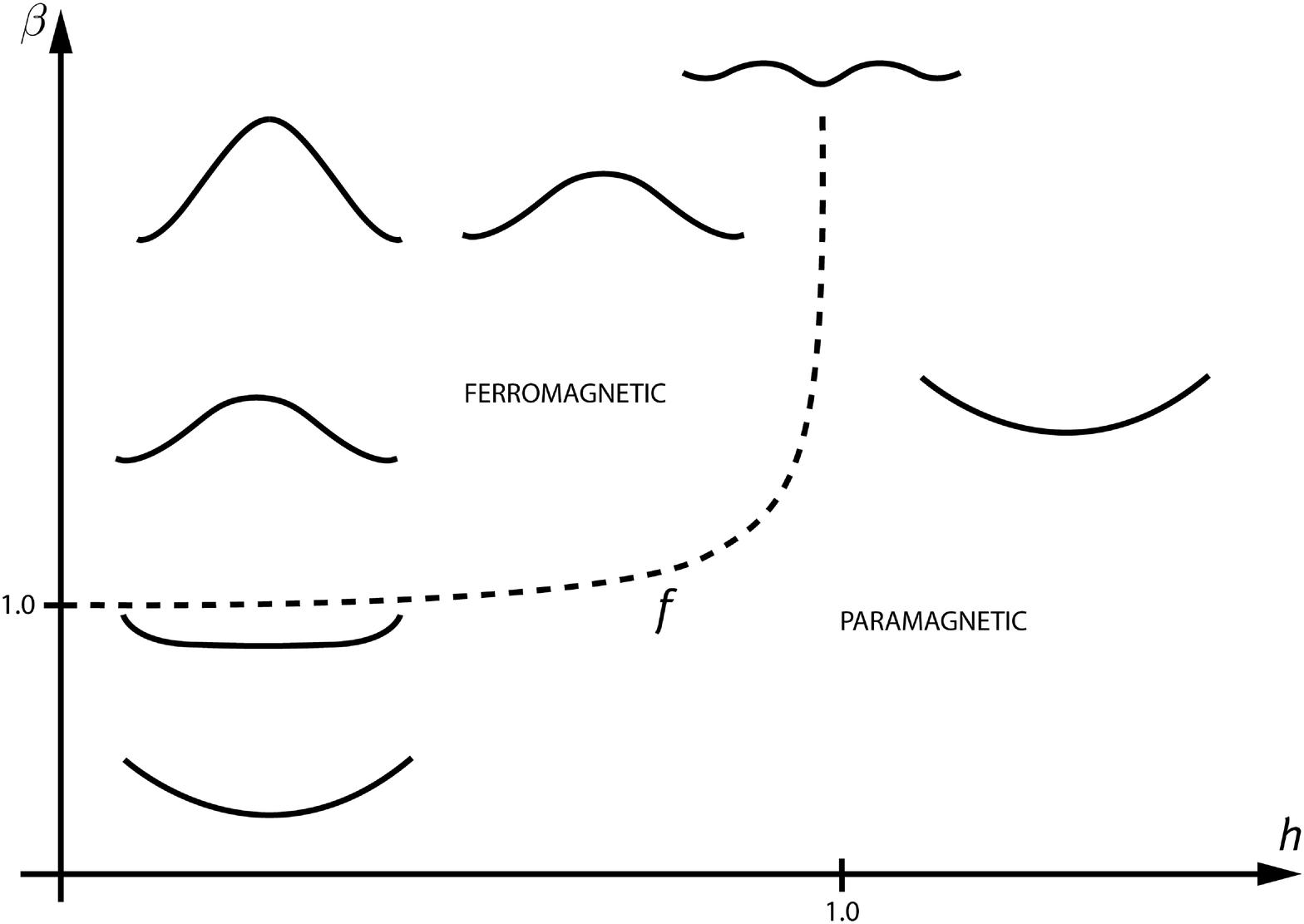}
\caption{\label{fig2} $G$ displays phase transitions as a function of $\beta$ and $h$.}
\end{center}
\end{figure}

\noindent
Very similar to what happened in the previous example, we now see (cf.\ Figure \ref{fig2}) that there exists an increasing function $f:[0,1) \to \mathbb{R}$ with $f(0)=1$ and $f(x_n) \nearrow \infty$ as $x_n \nearrow 1$ such that
\begin{itemize}
\item[(i)] $\beta < f(h)$:\\
$G$ has an unique global minimum at $x=0$ of type 1 and, thus, the system is in the paramagnetic phase.
\item[(ii)] $\beta > f(h)$:\\
$G$ has two symmetric global minima, both of type 1, and, thus,  the system is in the ferromagnetic phase.
\item[(iii)] $\beta = f(h)$:\\
On the critical line, the system shows a different behavior than the one displayed in the previous example. In particular, we do not observe a first-order phase transition.
\end{itemize}
\end{subsection}

%%%%%%%%%%%%%%%%%%%%%%%%%%%%%%%%%%%%%%%%%%%%%%%%%%%%%%%%%%%%%%%%%%%%%%%%%%
%\newpage

\vspace{24pt}
\paragraph{Acknowledgements.} The authors would like to thank Felix Lucka for an interesting discussion about Example \ref{lastexample}.
%The authors would like to thank C. K\"ulske for bringing the references \cite{Come1989}, \cite{Iaco2010} and \cite{Opoku2010} to their attention and for %his comments relating these references with the results in this paper.

\vspace{12pt}

%%%%%%%%%%%%%%%%%%%%%%%%%%%%%%%%%%%%%%%%%%%%%%%%%%%%%%%%%%%%%%%%%%%%%%%%%%
\bibliographystyle{alpha}			
{\footnotesize \bibliography{LDP}}
\label{chap:bib}
%%%%%%%%%%%%%%%%%%%%%%%%%%%%%%%%%%%%%%%%%%%%%%%%%%%%%%%%%%%%%%%%%%%%%%%%%%
\end{document}